\documentclass{article}
\usepackage{amssymb}
\usepackage[french,english]{babel}
\usepackage{mathrsfs}
\usepackage[colorlinks, urlcolor=blue, citecolor=black, linkcolor=black, breaklinks]{hyperref}
\usepackage{algorithm}
\usepackage{algpseudocode}

\algnewcommand{\LineComment}[1]{\State \(\triangleright\) #1}
\usepackage{mathtools}
\usepackage{etoolbox}
\usepackage{caption}
\usepackage{amsthm}
\usepackage{listings}
\usepackage{color} %red, green, blue, yellow, cyan, magenta, black, white
\definecolor{mygreen}{RGB}{28,172,0} % color values Red, Green, Blue
\definecolor{mylilas}{RGB}{170,55,241}
\newtheorem{theorem}{Theorem}[section]
\newtheorem{definition}{Definition}[section]

\newtheorem{proposition}{Proposition}[section]

\newtheorem{remark}{Remark}[section]

\newcommand{\F}{{\mathbb F}}

\def\og{\leavevmode\raise.3ex\hbox{$\scriptscriptstyle\langle\!\langle$~}}
\def\fg{\leavevmode\raise.3ex\hbox{~$\!\scriptscriptstyle\,\rangle\!\rangle$}}
\usepackage{authblk}
\title{Optimization of the scalar complexity of Chudnovsky$^2$ multiplication algorithms in finite fields}

\author[1]{St\'{e}phane Ballet}
\author[1]{Alexis Bonnecaze}
\author[1]{Thanh-Hung Dang}

\affil[1]{Aix-Marseille Univ, CNRS, Centrale Marseille, I2M, Marseille, France}

\begin{document}

\maketitle

\begin{abstract}
We propose several constructions for the original multiplication algorithm of D.V. and G.V. Chudnovsky 
in order to improve its scalar complexity. We highlight  the set of generic strategies who underlay 
the optimization of the scalar complexity, according to parameterizable criteria. As an example, 
we apply this analysis to the construction of type elliptic Chudnovsky$^2$ multiplication algorithms 
for small extensions. As a case study, we significantly improve the Baum-Shokrollahi 
construction for multiplication in $\mathbb F_{256}/\mathbb F_4$. 
%\keywords{algebraic function field, finite field, algorithmic complexity}
%\subclass{}
\end{abstract}

\section{Introduction}

\subsection{Context}

%%%%%%%%%%%%%%%%%%%%%%%%%%

The construction  of efficient arithmetic operation algorithms is still a problem of topicality. 
These algorithms are indeed heavily used in many domains of computer sciences or information theory.  
 It is important to conceive and develop efficient arithmetic algorithms combined with an optimal implementation method. 
 In this work, our interest lies in multiplication algorithms in any extension of finite field introduced in 1987 by  D.V. and G.V Chudnovsky \cite{chch} 
 and based upon interpolation on some algebraic curves defined over finite fields. 
 Our goal is to improve this method so that its complexity in terms of number of operations is optimized. 
 %%%%%%%%%%%%%%%%%%%%%%%%%%
 
 \medskip

More precisely, the complexity of a multiplication algorithm in ${\mathbb F}_{q^n}$ depends on the number
of multiplications and additions in $\F_q$. But here, we are particularly interested by the multiplicative complexity 
of multiplication in a finite field $\F_{q^n}$, i.e. by the number of multiplications in $\F_q$ required to multiply in the 
$\F_q$-vector space $\F_{q^n}$ of dimension $n$. There exist two types of multiplications in $\mathbb F_q$: 
the scalar multiplication and the  bilinear  one. The scalar multiplication is the multiplication by a non-trivial constant 
(i.e. not equal to $0$ or $1$) in $\mathbb F_{q}$, which does not depend on the elements of $\mathbb F_{q^n}$
that are multiplied. The bilinear multiplication is a multiplication that depends on the elements of $\mathbb F_{q^n}$ 
that are multiplied. The bilinear complexity is independent of the  chosen representation of the finite field.

 \medskip

%%%%%%%%%%%%%%%%%%%%%%%%%

Let $q$ be a prime power, ${\mathbb F}_q$ the finite field with $q$ elements and ${\mathbb F}_{q^n}$ the degree $n$ extension of ${\mathbb F}_q$.
If $\mathcal{B}=\{e_1,...,e_n\}$ is a basis of ${\mathbb F}_{q^n}$ over $\mathbb F_q$ then for $x=\sum_{i=1}^{n}x_ie_i$ and $y=\sum_{j=1}^{n}y_je_j$, we have the product
\begin{equation}\label{calculdirect}
z=xy=\sum_{h=1}^{n}z_he_h=\sum_{h=1}^{n}\biggr( \sum_{i,j=1}^{n}t_{ijh}x_iy_j\biggl)e_h,
\end{equation}
 where $e_ie_j=\sum_{h=1}^{n}t_{ijh}e_h,$ $t_{ijh}\in \mathbb F_q$ being some constants.
 
\medskip

Then, we see that the direct calculation of $z=(z_1,...,z_n)$ using (\ref{calculdirect}) {\it a priori} requires $n^2$ non-scalar multiplications $x_iy_j$, $n^3$  scalar multiplications and $n^3-n$ additions.

\begin{definition}\label{defmus}
The total number of scalar multiplications in $\F_q$ used in an algorithm ${\mathcal U}_{q,n}$ of multiplication in $\F_{q^n}$ is
called scalar complexity of ${\mathcal U}_{q,n}$ and denoted $\mu_s({\mathcal U}_{q,n})$.
\end{definition}

Moreover, the multiplication of two elements of $\mathbb F_{q^n}$ is an $\mathbb F_q$-bilinear map from 
$\mathbb F_{q^n} \times \mathbb F_{q^n}$ onto $\mathbb F_{q^n}$. Then, it can be considered as an 
$\mathbb F_q$-linear map from the tensor product ${\mathbb F_{q^n} \otimes_{\F_q} \F_{q^n}}$onto 
$\mathbb F_{q^n}$. Therefore, it can also be  considered as an element $T$ of 
$({{\F_{q^n}})^\star \otimes_{\F_q} ({\F_{q^n}})^\star \otimes_{\mathbb F_q} \mathbb F_{q^n}}$, 
where $\F_{q^n}^\star$ denotes the dual of $\F_{q^n}$.

\medskip

Set
$$T=\sum_{i=1}^{r} x_i^\star\otimes y_i^\star\otimes c_i,$$
where $x_i^\star\in\F_{q^n}^\star$, $y_i^\star\in\F_{q^n}^\star$
and $c_i\in\F_{q^n}$.
The following holds for any ${x,y \in \F_{q^n}}$: $$x\cdot y=T(x\otimes y)=\sum_{i=1}^r x_i^\star(x) y_i^\star(y) c_i.$$

\begin{definition}\label{defmubmum}
A multiplication algorithm ${\mathcal U}_{q,n}$ in $\F_{q^n}$
is an expression
$$
x\cdot y=\sum_{i=1}^r x_i^\star(x) y_i^\star(y) c_i,
$$
where $x^\star_i,y^\star_i \in ({\F_{q^n}})^\star$, and $c_i\in \F_{q^n}$.

The number $r$ of summands in this expression is called the bilinear complexity of the algorithm ${\mathcal U}_{q,n}$
and is denoted by $\mu_b({\mathcal U}_{q,n})$. The multiplicative complexity of ${\mathcal U}_{q,n}$ is $\mu_{m}({\mathcal U}_{q,n})= \mu_b({\mathcal U}_{q,n})+\mu_s({\mathcal U}_{q,n})$.
\end{definition}

\begin{definition}
The minimal number of summands in a decomposition of the tensor $T$ of the multiplication in $\F_{q^n}$
is called the bilinear complexity of the multiplication in $\F_{q^n}$ and is denoted by
$\mu_b(q,n)$:
$$
\mu_b(q,n) = \min_{{\mathcal U}} \mu_b({\mathcal U}) 
$$
where ${\mathcal U}$ is running over all bilinear multiplication algorithms in $\F_{q^n}$ over $\F_q$.
\end{definition}

\subsection{Some known results}
Let us recall some known results useful for this study.
 In their seminal papers, Winograd \cite{wino3} and De Groote \cite{groo} have shown that $\mu_b(q,n)\geq 2n-1$, 
 with equality holding if and only if $n \leq \frac{1}{2}q + 1$. Winograd has also proved \cite{wino3} that optimal 
 multiplication algorithms realizing the lower bound belong to the class of interpolation algorithms. 
 Later, generalizing interpolation algorithms on the projective line over $\mathbb F_q$ to algebraic curves of 
 higher genus over $\mathbb F_q$, D.V. and G.V. Chudnovsky provided a method \cite{chch} which enabled 
 to prove the \textit{linearity} \cite{ball1} of the bilinear complexity of multiplication in finite extensions of a finite field. 
 This is the so-called Chudnovsky$^2$ multiplication algorithm (or CCMA). 
 Applying CCMA with fitted elliptic curves, Shokrollahi in \cite{shok} (for the upper strict inequality) 
 and Chaumine in \cite{chau1} have shown that if 
 
\begin{equation}\label{ine}
\frac{1}{2}q +1< n \leq \frac{1}{2}(q+1+{\epsilon (q) })
\end{equation} 
where $\epsilon$ is the function defined by:
$$
\epsilon (q)=
\left \{
\begin{array}{l}
 \mbox{the greatest integer } \leq 2{\sqrt q} \mbox{ prime to $q$, if  $q$ is not a perfect square} \\
 2{\sqrt q},\mbox{ if $q$ is a perfect square,}
\end{array}
\right .
$$ 
then the bilinear complexity $\mu_b(q,n)$ of 
the multiplication in the finite extension $\F_{q^n}$ of the
finite field $\F_q$ is equal to $2n$. 

Then, many studies  focused on the qualitative improvement of CCMA with respect to the bilinear complexity (cf. \cite{bachpirararo}). 
But the problem of the optimization of its scalar complexity  has never been studied, although it was first raised in 2015 by 
Atighehchi, Ballet, Bonnecaze and Rolland \cite{atbaboro2} and so far it remained an open problem 
(cf. \cite[Open problem 10.2]{bachpirararo}). More explicitly, the structure of the involved matrices in CCMA
should be examined more closely but unfortunately, there are no theoretical means or criteria today to build 
the best matrices because they depend on the geometry of the curves, the field of definition of these curves, 
as well as the  involved Riemann-Roch spaces. The remaining open question is how to choose the geometrical objects, 
the associated Riemann-Roch vector-spaces as well as the suitable representation of those in order to minimise 
the number of zeros and 1 in the matrices of the evaluation maps involved in CCMA.

\subsection{New results and organization}

This article is the complete and generalized study of a preliminary introduction on the subject 
of scalar complexity, initiated in \cite{baboda}. Its main goal is to identify the set 
of fundamental generic strategies underlying the scalar complexity optimization (known as scalar optimization) 
of CCMA and the relevant quantities related to it. To do so, after having recalled in detail the CCMA method (cf. Section 
\ref{sectionMethodCCMA}) as well as the contextual framework (initial configuration) in which we are going to stand, 
we perform a detailed analysis (cf. Section \ref{anacomp}) of the scalar complexity ($\mu_{s}$) 
and the underlying relevant related quantities ($\mu_{s,0}$ and $\mu_{s,1}$). 

Then, in Section \ref{DQfixed}, which is the core of the paper, we present the general results allowing to identify 
the main lever (degree of freedom) of CCMA scalar optimization for a given CCMA algorithm. 
Then, from these results, we give two main generic strategies (Propositions \ref{optimisationalgocanonique1} 
and \ref{optimisationalgocanonique2}), whose optimization criteria can be parameterized (Remark \ref{Remark2Setup2}), 
which are the cornerstone of the complete strategy (see Section \ref{strategiecomplete}).
At this level (cf. Section \ref{DQfixed}), we give in particular the explicit presentation of the various corresponding 
optimization setup algorithms and lower bounds of the quantities $\mu_{s,0}(\mathcal{U}^{A}_{D,Q,\mathcal{P}})$ 
and $\mu_{s,0}(\mathcal{U}^{F,n}_{D,Q,\mathcal{P}})$.
In the complete strategy, we then show that the scalar complexity of the CCMA algorithm is independent 
of the order of the rational places to be evaluated, for a given set of rational places.
Finally, as an example, we specialize our study to elliptic CCMA algorithms, illustrated by two new designs 
of the Baum-Shokrollahi construction for multiplication in $\mathbb F_{256}/\mathbb F_4$ based on the elliptic 
Fermat curve $x^3+y^3=1$. These two new constructions, obtained by applying strategies guided by the optimization 
criterion of the number of zeros in the matrices involved, have scalar complexities significantly better than that of Baum-Shokrollahi.

%%%%%%%%%%%%%%%%%%%%%%%%%%%%%%%%

\section{The Chudnovsky$^2$ multiplication algorithm}

\subsection{Description and construction of CCMA}\label{sectionMethodCCMA}

Let $F/{\mathbb F}_q$ be an algebraic function field over the finite field ${\mathbb F}_q$ of genus $g(F)$. 
We denote by $N_k(F/{\mathbb F}_q)$ the number of places of degree $k$ of $F$ over ${\mathbb F}_q$. 
If $D$ is a divisor, ${\mathcal L}(D)$ denotes the Riemann-Roch space
associated to $D$. We denote by ${\mathcal O}_Q$ the valuation ring of the place $Q$ and by $F_Q $ 
its residue class field ${\mathcal O}_Q/Q$ which is isomorphic to $\F_{q^{\deg Q}}$ where $\deg Q$ 
is the degree of the place $Q$. The order of a divisor $D=\sum_{P}a_PP$ in the place $P$ is the number $a_P$, 
denoted $ord_P(D)$. The support of a divisor $D$ is the set $supp\hbox{ }D$ of the places $P$ such that $ord_P(D)\neq 0$. 
The divisor $D$ is called effective if $ord_P(D)\geq 0$ for any $P$. Let us define the classical Hadamard product $\odot$ in $\F_q^N$, 
where $N$ is a positive integer, by $(u_1,\ldots,u_{N}) \odot (v_1,\ldots,v_{N})=(u_1v_1,\ldots,u_{N}v_{N})$ for any $u_i, v_i$ in $\F_q$.
The following theorem describes the original multiplication algorithm of D.V. and G.V. Chudnovsky \cite{chch}.

\begin{theorem}\label{AlgoChud}
Let \begin{itemize}
      \item $n$ be a positive integer,
      \item $F /\F_{q}$ be an algebraic function field,
      \item $Q$ be a degree $n$ place of $F /\F_{q}$,
      \item $D$ be a divisor of $F /\F_{q}$,
      \item ${\mathcal P}=\{P_1,\ldots , P_N\}$ be an ordered set of places of degree one of $F /\F_{q}$.
   \end{itemize}

We suppose that $supp\;D \cap \{Q,P_1,...,P_N\}=\emptyset$ and that
\begin{enumerate}
\item[(i)] The evaluation map
$$\begin{array}{lccc}
 Ev_Q: & \mathcal L(D) & \rightarrow &   F_Q \\
    &  f              & \mapsto     & f(Q)
\end{array}$$
is surjective
\item[(ii)] The evaluation map
$$\begin{array}{lccl}
 Ev_{\mathcal P}: & \mathcal L(2D) & \rightarrow & \mathbb F_q^N \\
    &  f              & \mapsto     & \left(\strut f\left(P_1\right),\ldots,f\left(P_{N}\right)\right)
\end{array}$$
is injective
\end{enumerate}
Then
\begin{itemize}
\item [(1)]For any two elements $x$, $y$ in $\mathbb F_{q^n}$, we have a multiplication algorithm ${\mathcal U}_{q,n}$:

\begin{equation}\label{directproductalgoChud}
xy= E_Q \circ Ev_{\mathcal P}{|_{Im Ev_{\mathcal P}}}^{-1} \left( E_{\mathcal P}\circ Ev_Q^{-1}(x)\odot E_{\mathcal P}\circ Ev_Q^{-1}(y)  \right),
\end{equation}
where $E_Q$ denotes the canonical projection from the valuation ring ${\mathcal O}_Q$ of the place $Q$ in its residue class field $F_Q$, $E_{\mathcal P}$ the extension of $Ev_{\mathcal P}$ on the valuation ring ${\mathcal O}_Q$ of the place $Q$, $Ev_{\mathcal P}{|_{Im Ev_{\mathcal P}}}^{-1}$ the restriction of the inverse map of $Ev_{\mathcal P}$ on its image, and $\circ$ the standard composition map.

\item [(2)] We have: $$\mu_{b}({\mathcal U}_{q,n}) \leq N,$$ with equality if $N=dim\hbox{ } \mathcal L(2D)$.
\end{itemize}
\end{theorem}

Since $Q$ is a place of degree $n$, the residue class field $F_Q$ of place $Q$ is an extension of degree $n$ of $\mathbb F_{q}$ and it therefore can be identified to $\mathbb F_{q^n}$. Moreover, the evaluation map $Ev_Q$ being onto, one can associate the elements $x,y \in \mathbb F_{q^n}$ with elements of $\mathbb F_q$-vector space $\mathcal L(D)$, denoted respectively $f$ and $g$. We define $h:=fg$ by \begin{equation} \label{map h} (h(P_1),...,h(P_N))=
E_{\mathcal P}(f)\odot E_{\mathcal P}(g)=\left(f(P_1)g(P_1),...,f(P_N)g(P_N)\right). \end{equation}
We know that such an element $h$ belongs to $\mathcal L(2D)$ since the functions $f,g$ lie in $\mathcal L(D)$. Moreover, thanks to injectivity of $Ev_{\mathcal P}$, the function $h$ is in $\mathcal L(2D)$ and is uniquely determined by \eqref{map h}. We have
$$xy=Ev_Q(f)Ev_Q(g)=E_Q(h)$$

where $E_Q$ is the canonical projection from the valuation ring $\mathcal O_Q$ of the place Q in its residue class field $F_Q$, $Ev_Q$ is the restriction of $E_Q$ over the vector space $\mathcal L(D)$.

In order to make the study and the construction of this algorithm easier, we proceed in the following way. 
We choose a place $Q$ of degree $n$ and a divisor $D$ of degree $n+g-1$, such that $Ev_Q$ and 
$Ev_{\mathcal P}$ are isomorphisms. In this aim in \cite{ball1}, S. Ballet introduces simple numerical 
conditions on algebraic curves of an arbitrary genus $g$ giving a sufficient condition for the application 
of CCMA (existence of places of certain degree, of non-special divisors of degree $g-1$) generalizing the result of A. Shokrollahi
\cite{shok} for the elliptic curves. Let us recall this result:

\begin{theorem} \label{theoprinc}
Let $q$ be a prime power and let $n$ be an integer $>1$.
If there exists an algebraic function field $F/\F_q$ of genus $g$ satisfying the conditions

\begin{enumerate}
         \item  $N_n>0$
         (which is always the case if $2g+1 \leq q^{\frac{n-1}{2}}(q^{\frac{1}{2}}-1)$),
	\item $N_1 > 2n+2g-2$,
\end{enumerate}

then there exists a divisor $D$ of degree $n+g-1$ and a place $Q$ such that:

\begin{enumerate}
\item[(i)] The evaluation map
$$\begin{array}{lccc}
 Ev_Q: & \mathcal L(D) & \rightarrow & \frac {\mathcal O_Q}{Q}  \\
    &  f              & \mapsto     & f(Q)
\end{array}$$
is an isomorphism of vector spaces over $\F_q$.
\item[(ii)] There exist places $P_1$,...,$P_{N}$ such that the evaluation map
$$\begin{array}{lccl}
 Ev_{\mathcal P}: & \mathcal L(2D) & \rightarrow & \mathbb F_q^N \\
    &  f              & \mapsto     & \left(\strut f\left(P_1\right),\ldots,f\left(P_{N}\right)\right)
\end{array}$$
is an isomorphism of vector spaces over $\F_q$ with $N=2n+g-1$.
\end{enumerate}
\end{theorem}

\begin{remark}
First, note that in the elliptic case, the condition (2) is a large inequality  thanks to a result due to Chaumine \cite{chau1}.
Secondly, note also that the divisor $D$ is not necessarily effective.
\end{remark}

By this last remark, it is important to add the property of effectivity for the divisor $D$ in a perspective of implemention.
Indeed, it is easier to construct the algorithm CCMA with this assumption because in this case $\mathcal L(D)\subseteq \mathcal L(2D)$ and we can directly apply the evaluation map $Ev_{\mathcal P}$ instead of $E_{\mathcal P}$ in the algorithm (\ref{directproductalgoChud}), by means of a suitable representation of $\mathcal L(2D)$. Moreover, in this case we need to consider simultaneously the assumption that the support of the divisor $D$ does not contain the rational places and the place $Q$ of degree $n$ and the  assumption of effectivity of the divisor $D$. Indeed, it is known that the support moving technic (cf. \cite[Lemma 1.1.4.11]{piel}), which is a direct consequence of Strong Approximation Theorem (cf. \cite[Proof of Theorem I.6.4]{stic2}), applied on an effective divisor generates the loss of effectivity of the initial divisor (cf. also \cite[Remark 2.2]{atbaboro2}). So, let us suppose these two last assumptions.

\begin{remark}
As in \cite{ball2}, in practice, we take as a divisor $D$ one place of degree $n+g-1$. It has
the advantage to solve the problem of the support of divisor $D$ (cf. also \cite[Remark 2.2]{atbaboro2}) as well as the problem of the effectivity of the divisor D. However, it is not required to be considered in the theoretical study, but, as we will see, it will have some importance in the strategy of optimization.
\end{remark}

We can therefore consider the basis ${\mathcal B}_{Q}$ of the residue class field $F_Q$  over $\mathbb F_q$ as the image of a basis of $\mathcal L(D)$ by $Ev_Q$ or equivalently (which is sometimes useful following the considered situation) the basis of $\mathcal L(D)$ as the reciprocal image of a basis of the residue class field $F_Q$  over $\mathbb F_q$ by $Ev^{-1}_Q$. Let

\begin{equation} \label{BD}
{\mathcal B}_{D}:=\{f_1,...,f_n\}
\end{equation}

be a basis of $\mathcal L(D)$ and let us denote the basis of the supplementary space ${\mathcal M}$ of $\mathcal L(D)$ in $\mathcal L(2D)$ by

\begin{equation} \label{BCD}
{\mathcal B}^{c}_{D}:=\{f_{n+1},...,f_{N}\}
\end{equation}

where $N:= dim \mathcal L(2D)=2n+g-1$. Then, we choose

\begin{equation} \label{B2D}
{\mathcal B}_{2D}:={\mathcal B}_{D}\cup {\mathcal B}^{c}_{D}
\end{equation}

as the basis of $\mathcal L(2D)$.

We denote by $T_{2D}$ the matrix of the isomorphism $Ev_{\mathcal P}: \mathcal L(2D) \rightarrow \mathbb F_q^N$ in the basis ${\mathcal B}_{2D}$ of  $\mathcal L(2D)$ (the basis of $\F_q^N$ will always be the canonical basis). Then, we denote by $T_{D}$ the matrix of the first $n$ columns of the matrix $T_{2D}$. Therefore, $T_{D}$ is the matrix of the restriction of the  evaluation map $Ev_{\mathcal P}$ on the Riemann-Roch vector space $\mathcal L(D)$, which is an injective morphism.

Note that the canonical surjection $E_Q$ is the extension of the isomorphism $Ev_Q$ since, as $Q \notin supp(D)$, we have $\mathcal L(D)\subseteq \mathcal O_Q$. Moreover, as $supp(2D)= supp(D)$, we also have $\mathcal L(2D)\subseteq \mathcal O_Q$. We can therefore consider the images of elements of the basis ${\mathcal B}_{2D}$ by $E_Q$ and obtain a system of $N$ linear equations as follows: $$E_Q(f_r)=\sum_{m=1}^{n}c_r^mEv_Q(f_m), \;\;\;r=1,...,N$$
where $E_Q$ denotes the canonical projection from the valuation ring ${\mathcal O}_Q$ of the place $Q$ in its residue class field $F_Q$, $Ev_Q$ is the restriction of $E_Q$ over the vector space $\mathcal L(D)$ and $c_r^m \in \mathbb F_q$ for $r=1,...,N$.
Let $C$ be the matrix of the restriction of the map $E_Q$ on the Riemann-Roch vector space $\mathcal L(2D)$, from the basis ${\mathcal B}_{2D}$ in the basis ${\mathcal B}_{Q}$.
We obtain the product $z:=xy$ of two elements $x,y \in \mathbb F_{q^n}$ by the algorithm~(\ref{directproductalgoChud}) in Theorem \ref{AlgoChud},
where $M^t$ denotes the transposed matrix of the matrix $M$:

\begin{algorithm}[H]
\caption{Chudnovsky$^2$ Multiplication algorithm (CCMA) in $\mathbb F_{q^n}$}\label{algom}
\begin{algorithmic}
\Require
$x=\sum\limits_{i = 1}^n {x_iEv_Q(f_i)},$ and $y=\sum\limits_{i = 1}^n{y_iEv_Q(f_i)}$.\;\; 
\Ensure  $z=xy=\sum\limits_{i = 1}^n {z_iEv_Q(f_i)}$. \;\;

       \begin{enumerate}
         \item

         $X:=(X_1,...,X_{N})=(x_1,...,x_n)T_{D}^t=Ev_{\mathcal P}(x)$.

         $Y:=(Y_1,...,Y_{N})=(y_1,...,y_n)T_{D}^t=Ev_{\mathcal P}(y)$.

        \item $Z:=X\odot Y=(Z_1,...,Z_N)= (X_1Y_1,\ldots, X_{N}Y_{N})$.
        \item $(z_1,\ldots,z_n)=(Z_1,...,Z_N)(T_{2D}^t)^{-1}C^t=E_Q\circ Ev^{-1}_{\mathcal P}(Z)$.

       \end{enumerate}

\end{algorithmic}
\end{algorithm}
Now, we present an initial setup algorithm which is only done once.

\begin{algorithm}[H]
\caption{Setup algorithm of CCMA in $\mathbb F_{q^n}$}\label{algomsetup}
\begin{algorithmic}
\Require $F/{\mathbb F}_{q},~ Q,  D$, ${\mathcal P}=\{P_1,\ldots, P_{2n+g-1}\}$.
\Ensure  ${\mathcal B}_{2D}$, $T_{2D} \hbox{ and } CT_{2D}^{-1}.$
        \begin{enumerate}

          \item Check the function field $F/{\mathbb F}_{q}$, the place $Q$, the divisors $D$  are such that Conditions (i) and (ii) in Theorem \ref{theoprinc} can be satisfied.
          \item Represent $\F_{q^n}$  as the residue class field of the place $Q$.
          \item Construct a basis ${\mathcal B}_{2D}:=\{f_1, \ldots, f_n,f_{n+1}, \ldots, f_{2n+g-1}\}$ of $\mathcal L(2D)$, where ${\mathcal B}_{D}:=\{f_1, \ldots, f_n\}$
          is a basis of $\mathcal L(D)$, and ${\mathcal B}^{c}_{D}:=\{f_{n+1},...,f_{2n+g-1}\}$ a basis of the supplementary space
          ${\mathcal M}$ of $\mathcal L(D)$ in $\mathcal L(2D)$.
          \item Compute the matrices $T_{2D}$, $C$ and $CT_{2D}^{-1}$.
           \end{enumerate}
\end{algorithmic}
\end{algorithm}

\subsection{Complexity analysis}\label{anacomp}

Recall that the bilinear complexity of Chudnovsky$^2$ algorithms of type~(\ref{directproductalgoChud}) in 
Theorem \ref{AlgoChud} satisfying assumptions of Theorem \ref{theoprinc} is optimized. Therefore, we only focus 
on optimizing the scalar complexity of the algorithm. From Algorithm~\ref{algom} 
we observe that the number of scalar  multiplications depends directly on the number of zeros and 
of coefficients equal to $1$ in the matrices $T_{D}$ and $C.T^{-1}_{2D}$. Indeed, all the involved matrices 
being constructed once, the multiplication by a coefficient zero or $1$ in a matrix has not to be taken into account.
Let us give an algorithm $\mathcal{U}_{q,n}$ of type Algorithm \ref{algom} with a setup of type Algorithm \ref{algomsetup}. 
We can analyze the multiplicative complexity $\mu_m(\mathcal{U}_{q,n})$ of the algorithm $\mathcal{U}_{q,n}$, 
i.e. in terms of the total number of multiplications in $\mathbb F_q$, in the following way. We call $\mu_{s,0}(\mathcal{U}_{q,n})$ 
(resp. $\mu_{s,1}(\mathcal{U}_{q,n})$) the scalar complexity of the algorithm $\mathcal{U}_{q,n}$, 
taking into account uniquely the number of zeros $N_{z}(T_{D})$ (resp. the number of ones $N_{1}(T_{D})$) 
and $N_{z}(C.T^{-1}_{2D})$ (resp. $N_{1}(C.T^{-1}_{2D})$) respectively in the matrices $T_{D}$ and $CT_{2D}^{-1}$. Consequently, we clearly have 
\begin{equation}
\mu_{s}(\mathcal{U}_{q,n})\leq \mu_{s,0}(\mathcal{U}_{q,n}) 
\end{equation}

\begin{equation}
\mu_{s}(\mathcal{U}_{q,n})\leq \mu_{s,1}(\mathcal{U}_{q,n})
\end{equation}
and so 

\begin{equation}
\mu_{m}(\mathcal{U}_{q,n})\leq \mu_{s,0}(\mathcal{U}_{q,n})+ \mu_{b}(\mathcal{U}_{q,n}) 
\end{equation}

\begin{equation}
\mu_{m}(\mathcal{U}_{q,n})\leq \mu_{s,1}(\mathcal{U}_{q,n})+ \mu_{b}(\mathcal{U}_{q,n})
\end{equation}

by Definition \ref{defmubmum}.

%%%%%%%%%%%%%%%%

\medskip

The multiplicative complexity of the algorithm $\mathcal{U}_{q,n}$ is equal to $$\mu_{m}(\mathcal{U}_{q,n})=(3n+1)(2n+g-1),$$ 
including $$\mu_{s}(\mathcal{U}_{q,n})=3n(2n+g-1)$$ scalar multiplications in the least case.
More precisely, we get the formula to compute the number of scalar multiplications of this algorithm with 
respect to the number of zeros and $1$ of the involved matrices as follows:
$$\mu_s(\mathcal{U}_{q,n})= 2\biggr(n(2n+g-1)-N_{z}(T_{D})-N_1(T_{D})\biggl) + $$
\begin{equation} \label{scalar}
\biggr(n(2n+g-1)-N_{z}(C.T^{-1}_{2D}) -N_1(C.T^{-1}_{2D})\biggl)= 3n(2n+g-1)-N_{z}-N_1,
\end{equation}

where
\begin{equation} \label{Nzscalar}
N_{z}=2N_{z}(T_{D})+N_{z}(C.T^{-1}_{2D})
\end{equation}

and 

\begin{equation} \label{N1scalar}
N_{1}=2N_{1}(T_{D})+N_{1}(C.T^{-1}_{2D}).
\end{equation}

Moreover, we see in Algorithm \ref{algom} that all the scalar multiplications come from steps 1 and 3. 
Thus, for the analysis of the scalar complexity of any algorithm $\mathcal{U}_{q,n}$, we will distinguish 
the scalar complexities of steps 1 and 3 (resp. denoted $\mathcal{U}_{A}$ and $\mathcal{U}_{R}$) 
by respectively $\mu_s(\mathcal{U}_{A})$ and $\mu_s(\mathcal{U}_{R})$ which are by Formula (\ref{scalar}):

\begin{equation}\label{As}
\mu_s(\mathcal{U}_{A})=2\biggr(n(2n+g-1)-N_{z}(T_{D})-N_1(T_{D})\biggl)
\end{equation}
and 

 \begin{equation}\label{Rs}
 \mu_s(\mathcal{U}_{R})=\biggr(n(2n+g-1)-N_{z}(C.T^{-1}_{2D}) -N_1(C.T^{-1}_{2D})\biggl).
 \end{equation}
 
 We also will distinguish the scalar complexity of these steps of the algorithm, taking only into account 
 the number of zeros (resp. the number of 1). 
 Note that if we take into account the number of zeros (resp. the number of 1) in the step $\mathcal{U}_{A}$, then we take into 
 account the number of zeros (resp. the number of 1) in the step $\mathcal{U}_{R}$.
 Thus, we call $\mu_{s,0}(\mathcal{U}_{A})$ (resp. $\mu_{s,1}(\mathcal{U}_{A})$) and $\mu_{s,0}(\mathcal{U}_{R})$ (resp. $\mu_{s,1}(\mathcal{U}_{R})$) the quantities:
 
 \begin{equation}\label{Az} 
 \mu_{s,0}(\mathcal{U}_{A})= \mu_s(\mathcal{U}_{A}) \hbox{ with } N_1(T_D)=0 
 \end{equation}
 \begin{equation} \label{A1}
 \mu_{s,1}(\mathcal{U}_{A})= \mu_s(\mathcal{U}_{A}) \hbox{ with } N_z(T_D)=0 
 \end{equation}
 
 and 
 
 \begin{equation} 
 \mu_{s,0}(\mathcal{U}_{R})= \mu_s(\mathcal{U}_{R}) \hbox{ with } N_1(C.T^{-1}_{2D})=0,
 \end{equation}
 \begin{equation} 
 \mu_{s,1}(\mathcal{U}_{R})= \mu_s(\mathcal{U}_{R}) \hbox{ with } N_z(C.T^{-1}_{2D})=0.
 \end{equation} 
 Thus, we have: 
 
 \begin{equation}\label{ARz}
 \mu_{s,0}(\mathcal{U}_{q,n})= \mu_{s,0}(\mathcal{U}_{A})+\mu_{s,0}(\mathcal{U}_{R})=3n(2n+g-1)-N_z,
 \end{equation}
 and
 \begin{equation}\label{AR1}
 \mu_{s,1}(\mathcal{U}_{q,n})= \mu_{s,1}(\mathcal{U}_{A})+\mu_{s,1}(\mathcal{U}_{R})=3n(2n+g-1)-N_1.
 \end{equation}
 
 \begin{remark}\label{Addition1versus0}
 For the scalar complexity (i.e. the number of scalar multiplications), the coefficients $1$ and $0$ 
 play a symmetrical role. However, if we are looking at the additions, this role is no longer symmetrical 
 because the coefficients $1$ present in the matrices increase the number of additions 
 in the multiplication algorithm. Thus, from this point of view, it is in every interest to 
 favor the maximization of the number of zeros. It is for this reason in particular that this article 
 will give priority to the study of $\mu_{s,0}(\mathcal{U}_{q,n})$.
 \end{remark}
 
\section{Optimization of the scalar complexity} \label{optimisationscalar}

In this paper, we mainly focus on the optimization of the quantity  $\mu_{s,0}(\mathcal{U}_{q,n})$ 
introduced in Section \ref{anacomp}. In this sense, reducing the number of operations means finding 
an algebraic function field $F/\F_q$ having a genus $g$ as small as possible and a suitable set of 
divisor and places $(D,Q,\mathcal{P})$ with a good representation of the associated Riemann-Roch spaces, 
namely such that the matrices $T_{D}$ and $C.T^{-1}_{2D}$ are as hollow as possible 
(i.e. with a maximal number of zeros). Therefore, for a place $Q$ and a suitable divisor $D$, 
we seek the best possible representations of Riemann-Roch spaces $\mathcal L (D)$ and $\mathcal L (2D)$ 
to maximize mainly both parameters $N_{z}(T_{D})$ and $N_{z}(C.T^{-1}_{2D})$.

\subsection{Different types of generic strategy}\label{sectionDTGS}

\subsubsection{With fixed divisor and places}\label{DQfixed}

In this section, we consider the optimization of any algorithm $\mathcal{U}_{q,n}$ for a fixed suitable set 
of divisor and places $(D,Q,\mathcal{P})$ for a given algebraic function field $F/\F_q$ of genus $g$. 
Hence, according to Section \ref{anacomp}, we will denote here more precisely the algorithm $\mathcal{U}_{q,n}$ 
as well as the associated quantities $\mathcal{U}_{A}$ and $\mathcal{U}_{R}$ thanks to the following definition:

%So, let us give the following definition:

\begin{definition}\label{udq}
We call $\mathcal{U}^{F,n}_{D,Q,\mathcal{P}}:=(\mathcal{U}^{A}_{D,Q,\mathcal{P}},\mathcal{U}^{R}_{D,Q,\mathcal{P}})$ 
a Chudnovsky$^2$ multiplication algorithm of type (\ref{directproductalgoChud}) where 
$\mathcal{U}^{A}_{D,Q,\mathcal{P}}:=E_{\mathcal P}\circ Ev_Q^{-1}$ and 
$\mathcal{U}^{R}_{D,Q,\mathcal{P}}:=E_Q \circ Ev_{\mathcal P}{|_{Im Ev_{\mathcal P}}}^{-1}$, 
satisfying the assumptions of Theorem \ref{AlgoChud}. We will say that two algorithms are equal, 
and we will note: $\mathcal{U}^{F,n}_{D,Q,\mathcal{P}}=\mathcal{U}^{F,n}_{D',Q',\mathcal{P'}}$, 
if  $\mathcal{U}^{A}_{D,Q,\mathcal{P}}=\mathcal{U}^{A}_{D',Q',\mathcal{P}'}$ and 
$\mathcal{U}^{R}_{D,Q,\mathcal{P}}=\mathcal{U}^{R}_{D',Q',\mathcal{P}'}$.

\end{definition}

Note that in this case, this definition makes sense only if the bases of implied vector-spaces are fixed. 
So, we denote respectively by ${\mathcal B}_{Q}$, ${\mathcal B}_{D}$,
and ${\mathcal B}_{2D}$ the basis of the residue class field $F_Q$, and of Riemann-Roch vector-spaces 
${\mathcal L}(D)$, and ${\mathcal L}(2D)$ associated to $\mathcal{U}^{F,n}_{D,Q,\mathcal{P}}$.  
Note that the basis of the $\F_q$-vector space  $\F^N_q$ is the canonical basis, up to permutation. 
Then, we obtain the following result:

\begin{proposition}\label{algocanonique}
Let us consider an algorithm $\mathcal{U}^{F,n}_{D,Q,\mathcal{P}}$ such that the divisor $D$ 
is an effective divisor, $D-Q$ a non-special divisor of degree $g-1$, and such that the cardinal of the set
$\mathcal{P}$ is equal to the dimension of the Riemann-Roch space $\mathcal L (2D)$. 
Then we can choose the basis ${\mathcal B}_{2D}$ as (\ref{B2D}) and for any $\sigma$ in $GL_{\F_q}(2n+g-1)$, 
where $GL_{\F_q}(2n+g-1)$ denotes the linear group, we have
$$\mathcal{U}^{F,n}_{\sigma(D),Q,\mathcal{P}}=\mathcal{U}^{F,n}_{D,Q,\mathcal{P}}$$ 
where $\sigma(D)$ denotes the action of $\sigma$ on the basis ${\mathcal B}_{2D}$ of $\mathcal L (2D)$ 
in $\mathcal{U}^{F,n}_{D,Q,\mathcal{P}}$, with a fixed basis ${\mathcal B}_{Q}$ of the residue class field 
of the place $Q$ and ${\mathcal B}_{c}$ the canonical basis of $\F_q^{2n+g-1}$. In particular, the quantities 
$N_{z}(C.T^{-1}_{2D})$ and $N_{1}(C.T^{-1}_{2D})$ are constant under this action.

\end{proposition}

\begin{proof}
Let $E$, $F$ and $H$ be three vector spaces of finite dimension on a field $K$respectively equipped with the basis ${\mathcal B}_{E}$, ${\mathcal B}_{F}$ and ${\mathcal B}_{H}$.
Consider two morphisms $f$ and $h$ respectively defined from $E$ into $F$ and from $F$ into $H$ and consider respectively their associated matrix $M_f({\mathcal B}_{E},{\mathcal B}_{F})$ and $M_h({\mathcal B}_{F},{\mathcal B}_{H})$. Then it is obvious that the matrix $M_{h\circ f}({\mathcal B}_{E},{\mathcal B}_{H})$ of the morphism $h\circ f$ is independant from the choice of the basis ${\mathcal B}_{F}$ of $F$. As the divisor $D$ is effective, we have ${\mathcal L(D)}\subset {\mathcal L(2D)}$ and then $\mathcal{U}^{A}_{D,Q,\mathcal{P}}:=E_{\mathcal P}\circ Ev_Q^{-1}=Ev_{\mathcal P}\circ Ev_Q^{-1}$ and as $D-Q$ a non-special divisor of degree $g-1$, $Ev_Q$ is an isomorphism from $\mathcal L(D)$ into $F_Q$ and we have $\mathcal{U}^{A}_{D,Q,\mathcal{P}}=Ev_{\mathcal P}{|_{\mathcal L(D)}}\circ Ev_Q^{-1}$.
Moreover, as the cardinal of the set $\mathcal{P}$ is equal to the dimension of the Riemann-Roch space $\mathcal L (2D)$, $Ev_{\mathcal P}$ is an isomorphism from ${\mathcal L(2D)}$ into $\F_q^{2n+g-1}$ equipped with the canonical basis ${\mathcal B}_{c}$. Thus, $\mathcal{U}^{R}_{D,Q,\mathcal{P}}:=E_Q \circ Ev^{-1}_{\mathcal P}{|_{Im Ev_{\mathcal P}}}=E_{Q}|_{{\mathcal L(2D)}} \circ Ev^{-1}_{\mathcal P}$. Then, the matrix of $\mathcal{U}^{A}_{D,Q,\mathcal{P}}$ (resp. $\mathcal{U}^{R}_{D,Q,\mathcal{P}}$) is invariant under the action of $\sigma$ in $GL_{\F_q}(n)$ (resp. in $GL_{\F_q}(2n+g-1)$) on the basis ${\mathcal B}_{D}$ (resp. ${\mathcal B}_{2D}$) since the set $(E,F,H)$ is equal to $(F_Q,{\mathcal L(D)}, \F_q^{2n+g-1})$ (resp. $(\F_q^{2n+g-1},{\mathcal L(2D)}, F_{Q})$) for $h\circ f:=Ev_{\mathcal P}{|_{\mathcal L(D)}}\circ Ev_Q^{-1}$ (resp. $E_{Q}|_{{\mathcal L(2D)}} \circ Ev^{-1}_{\mathcal P}$).
\end{proof}

\medskip

Apart from the fact that this result provides a generic construction strategy 
of Chudnovsky's algorithm leading to significantly improve (and even optimize)  
the scalar complexity of this algorithm, this result also highlights a preferential configuration. 
Indeed, since the quantities $N_{z}(C.T^{-1}_{2D})$ and $N_{1}(C.T^{-1}_{2D})$ 
are constant under the action of the linear group, one has the choice, without consequence upon the scalar complexity, 
of the basis of the supplement ${\mathcal L}(D)$ in ${\mathcal L}(2D)$. Also, we favor 
a kernel-type configuration which not only has the particularity of having no negative impact on the 
global scalar complexity of the algorithm $\mathcal{U}=(\mathcal{U}_{A},\mathcal{U}_{R})$ 
but also to simplify the scalar optimization process of these algorithms as well as their use
in the return phase $\mathcal{U}_{R}$, the latter item having already been noticed in the context 
 of the exponentiation in \cite{atbaboro2}. Therefore, we need the following definition: 

\begin{definition}\label{defconstructionkernel}
Let $\mathcal{U}^{F,n}_{D,Q,\mathcal{P}}:=(\mathcal{U}^{A}_{D,Q,\mathcal{P}},\mathcal{U}^{R}_{D,Q,\mathcal{P}})$ 
be a Chudnovsky$^2$ multiplication algorithm in a finite field $\F_{q^n}$, satisfying the assumptions of 
Proposition \ref{algocanonique}. Then the algorithm $\mathcal{U}^{F,n}_{D,Q,\mathcal{P}}$ is said kernel-type 
if the basis ${\mathcal B}_{2D}$ of $\mathcal L (2D)$ used in the evaluation map  
$\mathcal{U}^{R}_{D,Q,\mathcal{P}}:=E_Q \circ Ev_{\mathcal P}^{-1}$ 
%$\mathcal{U}^{R}_{D,Q,\mathcal{P}}:=E_Q \circ Ev_{\mathcal P}{|_{Im Ev_{\mathcal P}}}^{-1}$ 
is such that $${\mathcal B}_{2D}={\mathcal B}_{D}\cup{\mathcal B}^{c}_{D},$$
where ${\mathcal B}_{D}$ is a basis of $\mathcal L (D)$ used in the evaluation map 
$\mathcal{U}^{A}_{D,Q,\mathcal{P}}:=E_{\mathcal P}|_{{\mathcal L(2D)}}\circ Ev_Q^{-1}$ and ${\mathcal B}^{c}_{D}$ 
is a basis of  the supplementary space ${\mathcal M}:=Ker E_{Q}|_{{\mathcal L(2D)}}$ of 
$\mathcal L(D)$ in $\mathcal L(2D)$. Any construction of a kernel-type algorithm $\mathcal{U}^{F,n}_{D,Q,\mathcal{P}}$ 
will be called a kernel-type construction.
\end{definition}

\medskip

%%%%%%%%%%%%%%%

\begin{proposition} \label{optimisationalgocanonique1}
Let $\mathcal{U}^{F,n}_{D,Q,\mathcal{P}}=(\mathcal{U}^{A}_{D,Q,\mathcal{P}},\mathcal{U}^{R}_{D,Q,\mathcal{P}})$ 
be a kernel-type Chudnovsky$^2$ multiplication algorithm in a finite field $\F_{q^n}$.
The optimal scalar complexity \\ $\mu^{opti}_{s,0}(\mathcal{U}^{A}_{D,Q,\mathcal{P}})$ of 
$\mathcal{U}^{F,n}_{D,Q,\mathcal{P}}$ is reached for the set $\{{\mathcal B}_{D,max}, {\mathcal B}_Q\}$  
such that ${\mathcal B}_{D,max}$ is a basis of $\mathcal L (D)$
satisfying $$N_{z}(T_{D,max})=\max_{\sigma \in GL_{\F_q}(n)}\{N_{z}(T_{\sigma(D)})\},$$ 
where $\sigma(D)$ denotes the action of
$\sigma$ on the basis ${\mathcal B}_{D}$ of $\mathcal L (D)$ in $\mathcal{U}^{F,n}_{D,Q,\mathcal{P}}$,  $T_{D,max}$ 
the matrix of the restriction of the  evaluation map $Ev_{\mathcal P}$ on the Riemann-Roch vector space $\mathcal L(D)$ 
equipped with the bases ${\mathcal B}_{D,max}$ and ${\mathcal B}_{Q}=Ev_Q({\mathcal B}_{D,max})$.
More precisely, we have 
\begin{multline*}\mu^{opti}_{s,0}(\mathcal{U}^{A}_{D,Q,\mathcal{P}})=\min_{\sigma \in GL_{\F_q}(n)} 
\{\mu_{s,0}(\mathcal{U}^{A}_{\sigma(D),Q,\mathcal{P}}) \mid \sigma({\mathcal B}_{D}) \hbox{ is the basis of } 
\mathcal L (D) \\ \hbox{ and } {\mathcal B}_Q=Ev_Q({\mathcal B}_{D})\}
\end{multline*}
$$=2\biggr(n(2n+g-1)-N_{z}(T_{D,max})\biggl).$$
Then, the scalar complexity of the algorithm $\mathcal{U}^{F,n}_{D,Q,\mathcal{P}}$ relatively to the basis ${\mathcal B}_{D,max}$ is:
$$\mu_{s,0}(\mathcal{U}_{D,Q,\mathcal{P}})=3n(2n+g-1)-\biggr(2N_{z}(T_{D,max})+N_{z}(T^{-1}_{2D,n})\biggl),$$
where  matrices $C$ and $T_{2D}$ are defined with respect to the basis ${\mathcal B}_Q=Ev_Q({\mathcal B}_{D,max})$, and ${\mathcal B}_{2D}= {\mathcal B}_{D,max}\cup {\mathcal B}^{c}_{D}$ and $T^{-1}_{2D,n}$ denotes the matrix made up of the $n$ first lines of the matrix  $T^{-1}_{2D}$.

\end{proposition}

\begin{proof}
The value of  $\mu^{opti}_{s,0}(\mathcal{U}^{A}_{D,Q,\mathcal{P}})$ follows directly from Proposition \ref{algocanonique} and formulae (\ref{As}) and (\ref{Az}). 
Then, the quantity $\mu_{s,0}(\mathcal{U}_{D,Q,\mathcal{P}})$ obtained with the basis ${\mathcal B}_{D,max}$ follows from formulae (\ref{scalar}) and (\ref{Nzscalar}). Note that since the algorithm $\mathcal{U}^{F,n}_{D,Q,\mathcal{P}}$ is kernel-type then we have $CT_{2D}^{-1}=T^{-1}_{2D,n}$ because ${\mathcal B}^{c}_{D}$ is a basis of the kernel of $E_{Q}|_{{\mathcal L(2D)}}$.

\end{proof}

\begin{proposition} \label{optimisationalgocanonique2}

Let $\mathcal{U}^{F,n}_{D,Q,\mathcal{P}}$ be a kernel-type Chudnovsky$^2$ multiplication algorithm in a finite field $\F_{q^n}$
such that ${\mathcal B}_{2D}= {\mathcal B}_{D}\cup {\mathcal B}^{c}_{D}$.
The optimal scalar complexity $\mu^{opti}_{s,0}(\mathcal{U}^{F,n}_{D,Q,\mathcal{P}})$ of $\mathcal{U}^{F,n}_{D,Q,\mathcal{P}}$ 
is reached for the set $\{{\mathcal B}_{D,max}, {\mathcal B}_Q\}$  such that 
${\mathcal B}_{D,max}$ is a basis of $\mathcal L (D)$ for which
 $2N_{z}(T_{\sigma(D)})+N_{z}(T^{-1}_{2\sigma(D),n})$ is maximal 
 under the action of
$\sigma\in GL_{\F_q}(n)$ on the basis ${\mathcal B}_{D}$ of $\mathcal L (D)$ of the matrix $T_{D}$ 
where $T_{\sigma(D)}$ (resp. $T^{-1}_{2\sigma(D),n}$) denotes the matrix $T_D$ (resp. the $n$ first lines 
of the matrix  $T^{-1}_{2D}$) in the basis $\sigma({\mathcal B}_{D})$ 
(resp. $\sigma({\mathcal B}_{D})\cup {\mathcal B}^{c}_{D}$) 
of $\mathcal L (D)$ (resp. $\mathcal L (2D)$),  and ${\mathcal B}_{Q}=Ev_Q({\mathcal B}_{D,max})$.
In particular, \begin{multline*}\mu^{opti}_{s,0}(\mathcal{U}^{F,n}_{D,Q,\mathcal{P}})=\min_{\sigma \in GL_{\F_q}(n)} \{\mu_{s,0}(\mathcal{U}^{F,n}_{\sigma(D),Q,\mathcal{P}}) \mid \sigma({\mathcal B}_{D}) \hbox{ is the basis of } \mathcal L (D) \\ \hbox{ and } {\mathcal B}_Q=Ev_Q({\mathcal B}_{D})\})\end{multline*}
$$=3n(2n+g-1)-N_{z,max},$$ where  
$$N_{z,max}=\max_{\sigma \in GL_{\F_q}(n)}\{2N_{z}(T_{\sigma(D)})+N_z(T^{-1}_{2\sigma(D),n})\},$$ and 
matrices $C$ and $T_{2D}$ are defined with respect to the basis ${\mathcal B}_Q=Ev_Q({\mathcal B}_{D,max})$, and ${\mathcal B}_{2D}= {\mathcal B}_{D,max}\cup {\mathcal B}^{c}_{D}$.

\end{proposition}

\begin{proof}
The value of  $\mu^{opti}_{s,0}(\mathcal{U}_{D,Q,\mathcal{P}})$ obtained with the basis ${\mathcal B}_{D,max}$ 
follows directly from Proposition \ref{algocanonique} and formulae (\ref{ARz}).  
Note that since the algorithm $\mathcal{U}^{F,n}_{D,Q,\mathcal{P}}$ is kernel-type then we have $CT_{2D}^{-1}=T^{-1}_{2D,n}$ 
because ${\mathcal B}^{c}_{D}$ is a basis of the kernel of $E_{Q}|_{{\mathcal L(2D)}}$.
\end{proof}

\begin{remark}\label{optimisationgeneral}
Note that in Proposition \ref{optimisationalgocanonique1}, we can establish a similar result for $\mu^{opti}_{s,1}(\mathcal{U}^{A}_{D,Q,\mathcal{P}})$ (or even better resp. $\mu^{opti}_{s}(\mathcal{U}^{A}_{D,Q,\mathcal{P}})$) by optimizing the quantity $N_{1}(T_{D})$ (resp. the quantity $N_{z}(T_{D})+N_{1}(T_{D})$) instead of $N_{z}(T_{D}$). In the same way, in Proposition \ref{optimisationalgocanonique2},
we can establish a similar result for $\mu^{opti}_{s,1}(\mathcal{U}^{F,n}_{D,Q,\mathcal{P}})$ (or even better resp. $\mu^{opti}_{s}(\mathcal{U}^{F,n}_{D,Q,\mathcal{P}})$) by optimizing the quantity $2N_{1}(T_{D})+N_{1}(T^{-1}_{2D,n})$ (resp. the quantity $2(N_{z}(T_{D})+N_{1}(T_{D}))+N_{z}(T^{-1}_{2D,n})+N_{1}(T^{-1}_{2D,n})$) instead of $2N_{z}(T_{D})+N_{z}(T^{-1}_{2D,n})$.
\end{remark}

Now, from these two previous results, we can highlight several strategies to improve the scalar complexity. 
The complexities of these strategies are clearly different. Therefore, the use of this or that strategy 
may be useful depending on the constraints to which we are subject.
New setup algorithms can be obtained directly from the analysis developed in Section \ref{DQfixed}.
More precisely, the following setup corresponds to the optimization of $\mu_{s,0}(\mathcal{U}^{A}_{D,Q,\mathcal{P}})$ described by Proposition  \ref{optimisationalgocanonique1}.

\begin{algorithm}[H]
\caption{New setup algorithm of CCMA in $\F_{q^n}$ from Proposition \ref{optimisationalgocanonique1}}\label{setup1}
\begin{algorithmic}
\Require $F/{\mathbb F}_{q},~ Q,  D$, ${\mathcal P}=\{P_1,\ldots, P_{2n+g-1}\}$.
\Ensure  ${\mathcal B}_{2D}$, ${\mathcal B}_Q$, $T_{2D} \hbox{ and } T_{2D,n}^{-1}.$
        \begin{enumerate}
          \item Check the function field $F/{\mathbb F}_{q}$, the place $Q$, the divisors $D$  are such that Conditions (i) 
          and (ii) in Theorem \ref{theoprinc} can be satisfied.
          \item Take an initial basis ${\mathcal B}_{D,0}$ for $\mathcal L(D)$ and construct a basis ${\mathcal B}^{c}_{D}:=\{f_{n+1},...,f_{2n+g-1}\}$ 
          of the supplementary space ${\mathcal M}:=Ker E_{Q}|_{{\mathcal L(2D)}}$ of $\mathcal L(D)$ in $\mathcal L(2D)$. 
          \item Go through the set (or subset) of bases ${\mathcal B}_{D}$ of $\mathcal L(D)$ from ${\mathcal B}_{D,0}$ and linear group $GL_q(n)$ in order to   
          compute $T_D$ and $N_z(T_D)$.
          \item
          Choose a basis ${\mathcal B}_{D}:=\{f_1, \ldots, f_n\}$ such that the matrix $T_{D}$ owns the largest number of zeros 
          (i.e. such that ${\mathcal B}_{D}:={\mathcal B}_{D,max}$ and $T_{D}:=T_{D,max}$).
           \item Set ${\mathcal B}_Q:=Ev_Q({\mathcal B}_{D,max})$ and ${\mathcal B}_{2D}:={\mathcal B}_{D,max}\cup{\mathcal B}^{c}_{D}$.
          \item Compute the matrices $T_{2D} \hbox{ and } T_{2D,n}^{-1}$ in the basis ${\mathcal B}_{2D}$.
           \end{enumerate}
\end{algorithmic}
\end{algorithm}

In the same way, from Proposition  \ref{optimisationalgocanonique2}, we can obtain the following new setup corresponding 
to the optimization of $\mu_{s,0}(\mathcal{U}^{F,n}_{D,Q,\mathcal{P}})$.

\begin{algorithm}[H]
\caption{New setup algorithm of CCMA in $\F_{q^n}$ from Proposition \ref{optimisationalgocanonique2}}\label{setup2}
\begin{algorithmic}
\Require $F/{\mathbb F}_{q},~ Q,  D$, ${\mathcal P}=\{P_1,\ldots, P_{2n+g-1}\}$.
 \Ensure ${\mathcal B}_{2D}$, ${\mathcal B}_Q$, $T_{2D} \hbox{ and } T_{2D,n}^{-1}.$

        \begin{enumerate}
          \item Check the function field $F/{\mathbb F}_{q}$, the place $Q$, the divisors $D$  are such that Conditions (i) and (ii) in Theorem \ref{theoprinc} can be satisfied.
          \item Take an initial basis ${\mathcal B}_{D,0}$ for $\mathcal L(D)$ and construct a basis ${\mathcal B}^{c}_{D}:=\{f_{n+1},...,f_{2n+g-1}\}$ 
          of the supplementary space ${\mathcal M}:=Ker E_{Q}|_{{\mathcal L(2D)}}$ of $\mathcal L(D)$ in $\mathcal L(2D)$. 
          \item Go through the set (or subset) of bases ${\mathcal B}_{D}$ of $\mathcal L(D)$ from ${\mathcal B}_{D,0}$ and linear group $GL_q(n)$ in order to   
          compute $T_D$ (resp. $T^{-1}_{2D,n}$ with ${\mathcal B}_{2D}={\mathcal B}_{D}\cup{\mathcal B}^{c}_{D}$) and $N_z(T_D)$ (resp. $N_z(T^{-1}_{2D,n})$).
          \item
          Choose a basis ${\mathcal B}_{D}:=\{f_1, \ldots, f_n\}$ such that $2N_z(T_{D})+N_z(T^{-1}_{2D,n})$ is the largest possible 
          (i.e. such that ${\mathcal B}_{D}:={\mathcal B}_{D,max}$ and $2N_z(T_{D})+N_z(T^{-1}_{2D,n}):=N_{z,max}$).

           \item Set ${\mathcal B}_Q:=Ev_Q({\mathcal B}_{D,max})$ and ${\mathcal B}_{2D}:={\mathcal B}_{D,max}\cup{\mathcal B}^{c}_{D}$.
           \end{enumerate}
\end{algorithmic}
\end{algorithm}

\begin{remark}\label{Remark1Setup2}
Note that in Algorithm \ref{setup2}, Step 6 of the Algorithm \ref{setup1} was performed in steps 3  and 4
since in order to construct $T^{-1}_{2D,n}$, we required to construct before the matrix $T_{2D}$.
\end{remark}

\begin{remark}\label{Remark2Setup2}
Note that in the setup algorithm \ref{setup1}, the steps 3 and 4 may be substituted by : 
choose a basis ${\mathcal B}_{D}:=(f_1, \ldots, f_n)$ such that the matrix $T_{D}$ owns 
the largest number of 1 or the largest number of 0 or 1 taken together, in the same spirit 
as Remark \ref{optimisationgeneral}. So, in the setup algorithm \ref{setup2}, the number $N_z$ 
in the steps 3 and 4 may be substituted by the number $N_1$ resp. $N_z+N_1$ 
for each matrice $T_D$ and $T^{-1}_{2D,n}$. However, we have chosen in this paper 
to focus particularly on  the number of zeros because it is possible to give an upper 
bound on this value, as we will see below.
\end{remark}

\vspace{0.3cm}

Indeed, let $F/\F_q$ be an algebraic function field of genus $g$ and  let ${\mathcal P}=\{P_1,...,P_N\}$ be an ordered set of  
pairwise distinct places of degree one in $F/\F_q$. Let us adapt slightly the notation used in \cite{stic2} to be homogeneous with the notation used in the description of CCMA. 
So, we consider that the divisors $G=P_1+\cdots+P_N$ and $D$ are divisors of $F/\F_q$ such that $suppG \cap suppD=\emptyset$. The algebraic geometry code (or Goppa code) $C_{\mathcal{L}}(G,D)$ associated with the divisors $G$ and $D$ is defined as
$$C_{\mathcal{L}}(G,D):=\{(f(P_1),...,f(P_N)) | f\in \mathcal{L}(D)\} \subseteq \F_q^N.$$
Then $C_{\mathcal{L}}(G,D)$ is an $[N,k,d]$ code with parameters $k=\dim\mathcal{L}(D)-\dim\mathcal{L}(D-G)$ and minimum distance $d \geq N-\deg D$ 
by \cite[Theorem 2.2.2]{stic2}. If $\{f_1,...,f_k\}$ is a basis of $\mathcal{L}(D)$, then by \cite[Corollary 2.2.3]{stic2} we have the following generator matrix for $C_{\mathcal{L}}(G,D)$
$$M:=\left(
  \begin{array}{ccc}
    f_1(P_1) & \cdots & f_1(P_N) \\
    f_2(P_1) & \cdots & f_2(P_N) \\
    \vdots & \vdots & \vdots \\
    f_k(P_1)& \cdots & f_k(P_N)\\
  \end{array}
\right).$$

In the Chudnovsky$^2$ multiplication algorithm (CCMA) defined in the context of Theorem \ref{theoprinc}, we consider the bijective evaluation map
$$\begin{array}{lccl}
 Ev_{\mathcal P}: & \mathcal L(2D) & \rightarrow & \mathbb F_q^N. \\
    &  f              & \mapsto     & \left(\strut f\left(P_1\right),\ldots,f\left(P_N\right)\right)
\end{array}$$
where $\deg D=n+g-1$ and $N=2n+g-1$.
Moreover, we recall that our construction of CCMA is made with the assumptions of Proposition \ref{algocanonique}, hence $\mathcal{L}(D)\subseteq \mathcal{L}(2D)$ since $D$ is an effective divisor. Then, the image of the restriction $Ev_{\mathcal{P}}|_{\mathcal{L}(D)}$ of $Ev_{\mathcal{P}}$ on $\mathcal{L}(D)$  is a 
$\F_q$-vector subspace of $\F_q^N$ of dimension $n$ 
 which can be seen as an algebraic geometry code $C_{\mathcal{L}}(G,D)=[N,n,d]$ where $G=P_1+\cdots+P_N$. 
 Therefore, we can prove the following results.

\begin{proposition}\label{propositioncode}
Let $\mathcal{U}^{F,n}_{D,Q,\mathcal{P}}$ be a Chudnovsky$^2$ multiplication algorithm in a finite field $\F_{q^n}$, satisfying the assumptions of Proposition \ref{optimisationalgocanonique1}.
Then we have: $$N_z(T_D)\leq n(n+g-1).$$ 
\end{proposition}

\begin{proof}
The matrix $T_D$ is such that
\begin{equation}\label{Nz_TD}
N_z(T_D)=n\cdot N - N_{nz}(T_D),
\end{equation}
where $N_{nz}(T_D)$ denotes the number of non-zero entries of $T_D$ and $N=2n+g-1$.
Moreover,  as $Ev_{\mathcal{P}}(\mathcal{L}(D))$ is an algebraic geometry code $C_{\mathcal{L}}(G,D)=[N,n,d]$ where $G=P_1+\cdots+P_N$, 
then $T_D^t$ is a generator matrix of this code. So, we have 
\begin{equation}\label{Nnz_TD}
N_{nz}(T_D) \geq n\cdot d,
\end{equation}
by the definition of the minimal distance of a code. Moreover, we have 
\begin{equation}\label{dAGC}
d\geq N - \deg D
\end{equation}
 by \cite[Theorem 2.2.2]{stic2}. 
So, we obtain by (\ref{Nz_TD}), (\ref{Nnz_TD}) and (\ref{dAGC}): 

\begin{equation}\label{upperbound_NzTD}
     N_z(T_D)\leq n\cdot\deg D
\end{equation}
As $\deg D=n+g-1$, we obtain the result.

\end{proof}

\begin{theorem}\label{corocode}
Let $\mathcal{U}^{F,n}_{D,Q,\mathcal{P}}=(\mathcal{U}^{A}_{D,Q,\mathcal{P}},\mathcal{U}^{R}_{D,Q,\mathcal{P}})$ be a Chudnovsky$^2$ multiplication algorithm in a finite field $\F_{q^n}$, satisfying the assumptions of Proposition \ref{optimisationalgocanonique1}.
Then we have: $$\mu_{s,0}(\mathcal{U}^{A}_{D,Q,\mathcal{P}}) \geq 2n^2$$ and $$\mu_{s,0}(\mathcal{U}^{F,n}_{D,Q,\mathcal{P}}) > 2n^2.$$

\end{theorem}

\begin{proof}
By Equalities (\ref{As}) and (\ref{Az}), we have $$\mu_{s,0}(\mathcal{U}^{A}_{D,Q,\mathcal{P}})=2(n(2n+g-1)-N_z(T_D).$$
Then, since $N_z(T_D)\leq n(n+g-1)$ by Proposition \ref{propositioncode}, we deduce the first inequality. 
Moreover, we have the trivial bound $N_z(T^{-1}_{2D,n})< n(2n+g-1)$. Thus, as 
$\mu_{s,0}(\mathcal{U}^{F,n}_{D,Q,\mathcal{P}})=3n(2n+g-1)-N_z=3n(2n+g-1)-2N_z(T_D)-N_z(T^{-1}_{2D,n})$ by Equality (\ref{ARz}), we obtain 
$\mu_{s,0}(\mathcal{U}^{F,n}_{D,Q,\mathcal{P}})>2n^2$.

\end{proof}

Now, we can give an optimization using a criterium obtained from Proposition \ref{propositioncode}.

\begin{algorithm}[H]
\caption{New setup algorithm of CCMA in $\F_{q^n}$ from Proposition \ref{propositioncode}} \label{setup3}
\begin{algorithmic}
\Require $F/{\mathbb F}_{q},~ Q,  D$, ${\mathcal P}=\{P_1,\ldots, P_{2n+g-1}\}$.
\Ensure  ${\mathcal B}_{2D}$, ${\mathcal B}_Q$, $T_{2D} \hbox{ and } T_{2D,n}^{-1}.$
        \begin{enumerate}
          \item Check the function field $F/{\mathbb F}_{q}$, the place $Q$, the divisors $D$  are such that Conditions (i) and (ii) in Theorem \ref{theoprinc} can be satisfied.
          \item Take an initial basis ${\mathcal B}_{D,0}$ for $\mathcal L(D)$ and construct a basis ${\mathcal B}^{c}_{D}:=\{f_{n+1},...,f_{2n+g-1}\}$ 
          of the supplementary space ${\mathcal M}:=Ker E_{Q}|_{{\mathcal L(2D)}}$ of $\mathcal L(D)$ in $\mathcal L(2D)$. 
          \item Go through the set (or subset) of bases ${\mathcal B}_{D}$ of $\mathcal L(D)$ from ${\mathcal B}_{D,0}$ and linear group $GL_q(n)$ 
          in order to compute $T_D$ and to construct the set $mB_D=\{ {\mathcal B}_{D}\mid N_z(T_D)=n(n+g-1)\}$.
          \item
          Choose a basis ${\mathcal B}_{D}:=\{f_1, \ldots, f_n\}\in mB_D$ such that $N_z(T^{-1}_{2D,n})$ is the largest possible. 
           \item Set ${\mathcal B}_Q:=Ev_Q({\mathcal B}_{D,max})$ and ${\mathcal B}_{2D}:={\mathcal B}_{D,max}\cup{\mathcal B}^{c}_{D}$.
           \item Compute the matrices $T_{2D} \hbox{ and } T_{2D,n}^{-1}$ in the basis 
          ${\mathcal B}_{2D}$.
           \end{enumerate}
\end{algorithmic}
\end{algorithm}

\begin{remark}\label{optimizationcriteriumNzthenN1}
Note that in the setup algorithm \ref{setup3}, the step 4 may be substituted by the best following criterium: choose a basis ${\mathcal B}_{D}\in mB_D$ such that $2N_1(T_{D})+N_z(T^{-1}_{2D,n})+N_1(T^{-1}_{2D,n})$ is the largest possible.
\end{remark}

\begin{remark}\label{FinalRemarkGenericStrategies}
As one can see, the algorithms proposed in this section are generic and in this sense they are well automatized for any set $(q,n,F/\F_q,D, Q)$.
Indeed the complexity of the optimization increases with the cardinal of $GL_q(n)$.
However, this complexity of optimization (although not having currently an accurate estimate) is 
much lower than that of a brute force optimization where all the bases of each of the vector spaces 
involved in the two linear applications must be tested. In fact, the strong point of the analysis 
conducted in this section is that it shows that the only relevant lever to optimize the CCMA algorithm concerns 
the representation of the $\mathcal L(D)$ space and only this space. Therefore, most of 
the complexity of this optimization lies in running over the linear group (or a subset) underlying this space, 
as well as in related operations.
\end{remark}

\subsubsection{Strategy of complete optimization}\label{strategiecomplete}

In the view of a complete optimization (with respect to scalar complexity i.e. with fixed bilinear complexity) 
of the multiplication in a finite field $\F_{q^n}$ by a Chudnovsky$^2$ multiplication algorithm, 
we have to vary the eligible sets $(F,D,Q,\mathcal{P})$.
We can vary the couples $(D,Q)$ 
satisfying the assumptions of Proposition \ref{algocanonique} 
and apply complete optimization Algorithm \ref{setup2} (or Algorithm \ref{setup2} with optimization 
criterium $N_1$ resp. $N_z+N_1$ as mentioned in Remark \ref{Remark2Setup2}): for instance, we can start by fixing the place $Q$ 
and then vary the suitable divisors $D$. 
Concerning the set $\mathcal{P}$ of rational places, we can show that two algorithms which differ only 
by the order of the places on which we evaluate have the same scalar complexity.
That is to say, for any permutation $\pi$ of the set $\mathcal{P}$, we wonder whether $\mathcal{U}^{F,n}_{D,Q,\pi(\mathcal{P})}$ is different from 
$\mathcal{U}^{F,n}_{D,Q,\mathcal{P}}$ in order to answer to the open problem mentioned in \cite[Remark 3]{baboda}. 
The action of $\pi$ corresponds to  a permutation of the canonical basis $\mathcal{B}_c$ of $\F_q^{2n+g-1}$. 
It corresponds to a permutation of the rows of the matrix $T_{2D}$. In this case, $N_{z}(T_{D})$ and $N_{1}(T_{D})$ 
are  obviously constant under the action of $\pi$. The following proposition also enables us to claim that $N_{z}(C.T^{-1}_{2D})$ 
and $N_{1}(C.T^{-1}_{2D})$ are constant under the action of $\pi$. 

\begin{proposition}
Let us consider an algorithm $\mathcal{U}^{F,n}_{D,Q,\mathcal{P}}$ such that $D$ is an effective divisor, 
$D-Q$ a non-special divisor of degree $g-1$, and $|\mathcal{P}|=\dim \mathcal{L}(2D)=N$.

Then for any $\pi$ in $S_N$ where $S_N$ is the symmetric group on the set $\{1,2,...,N\}$, we have $$\mu_s(\mathcal{U}^{F,n}_{D,Q,\mathcal{P}})= \mu_s(\mathcal{U}^{F,n}_{D,Q,\mathcal{\pi(P)}})$$ and 
$$\mu_{s,0}(\mathcal{U}^{F,n}_{D,Q,\mathcal{P}})= \mu_{s,0}(\mathcal{U}^{F,n}_{D,Q,\mathcal{\pi(P)}}).$$
In particular, the quantities $N_{z}(T_D)$ (resp. $N_{1}(T_D)$) and $N_{z}(CT^{-1}_{2D})$ (resp. $N_{1}(CT^{-1}_{2D})$) 
are constants under the action $\pi$.
\end{proposition}

\begin{proof}
Let $\mathcal{P}:=\{P_1,P_2,...,P_N\}$ be the ordered set of $N$ rational places used in the algorithm $\mathcal{U}^{F,n}_{D,Q,\mathcal{P}}$. 
We consider the action of the permutation $\pi\in S_N$ on the set $\mathcal{P}$ by setting 
$\mathcal{P}'=\pi . \mathcal{P}=\{P_{\pi(1)},P_{\pi(2)},...,P_{\pi(N)}\}$.

\medskip

Given a basis $\mathcal{B}_{2D}$ of Riemann-Roch space $\mathcal{L}(2D)$, we consider two evaluation maps:
\begin{equation} \label{EvP1}
   \begin{array}{ccl}
Ev_{\mathcal{P}}:  \mathcal L(2D) & \rightarrow & \mathbb{F}_q^N \\
      f              & \mapsto     & \big ( f(P_1),...,f(P_N)\big )
\end{array}
\end{equation}
and
\begin{equation} \label{EvP2}
   \begin{array}{ccl}
Ev_{\mathcal{P}'}:  \mathcal L(2D) & \rightarrow & \mathbb{F}_q^N \\
      f              & \mapsto     & \big ( f(P_{\pi(1)}),...,f(P_{\pi(N)})\big )
\end{array}
\end{equation}
We denote $\mathcal{B}_{\mathbb{F}_q^N}^c=(e_1,...,e_N)$ the canonical basis of $\mathbb{F}_q^N$ in (\ref{EvP1}) and $\mathcal{B}_{\mathbb{F}_q^N}^{\pi}=(e_{\pi(1)},...,e_{\pi(N)})$ the basis of $\mathbb{F}_q^N$ in (\ref{EvP2}).

Let us define an isomorphism $p: \mathbb{F}_q^N  \rightarrow \mathbb{F}_q^N$ by $p(e_i)=e_{\pi(i)}$ for $i=1..N$. The matrix representation of this map is denoted by $P$. We see that $P$ is a permutation matrix and note that $P^{-1}=P^t$. We have $$Ev_{\mathcal{P}'}=p\circ Ev_{\mathcal{P}}.$$
Then \begin{equation}\label{Aller_permu}
        \mathcal{U}^A_{D,Q,\mathcal{P}'}=E_{\mathcal{P}'}\circ Ev_{Q}^{-1} =p\circ E_{\mathcal{P}}\circ Ev_{Q}^{-1}=p\circ \mathcal{U}^A_{D,Q,\mathcal{P}}
      \end{equation}
and
\begin{equation}\label{Retour_permu}
        \mathcal{U}^R_{D,Q,\mathcal{P}'}=E_{Q}\circ Ev_{\mathcal{P}'}^{-1}|_{Im( Ev_{\mathcal P'})} =E_{Q}\circ \big( p\circ Ev_{\mathcal{P}}|_{Im(Ev_{\mathcal P})} \big )^{-1} =  \mathcal{U}^R_{D,Q,\mathcal{P}}\circ p^{-1}.
      \end{equation}

Observing the changement of the positions of rows of $T_{2D}$ and columns of $CT_{2D}^{-1}$ affected by (\ref{Aller_permu}) and (\ref{Retour_permu}) respectively, we have $N_{z}(T_D)$ (resp. $N_{1}(T_D)$) and $N_{z}(CT_{2D}^{-1})$ (resp. $N_{1}(CT_{2D}^{-1})$) are constants for any $\pi \in S_N$.

 By (\ref{scalar}) we obtain $$\mu_s(\mathcal{U}^{F,n}_{D,Q,\mathcal{P}})= \mu_s(\mathcal{U}^{F,n}_{D,Q,\mathcal{\pi(P)}})$$ and 
 $$\mu_{s,0}(\mathcal{U}^{F,n}_{D,Q,\mathcal{P}})= \mu_{s,0}(\mathcal{U}^{F,n}_{D,Q,\mathcal{\pi(P)}})$$
 for any $\pi\in S_N.$
\end{proof}

%\cite{serr5}

Finally, we can then look for a fixed suitable algebraic function field of genus $g$, 
up to isomorphism, and repeat all the previous steps. Moreover, it is still possible to look at 
the trade-off between scalar complexity and bilinear complexity by increasing the genus and 
then re-conducting all the previous optimizations (i.e. we take algebraic function fields with a genus 
larger than required for multiplying in $\F_{q^n})$.

\subsection{Optimization of scalar complexity in the elliptic case} \label{optimisationscalarelliptique}

Now, we study a specialisation of the Chudnovsky$^2$ multiplication algorithm of type (\ref{directproductalgoChud}) 
in the case of the elliptic curves (cf. inequality (\ref{ine})). In particular, we improve the effective algorithm constructed in the article of U. Baum and M.A. Shokrollahi \cite{bash} which presented an optimal algorithm from the point of view of the bilinear complexity in the case of the multiplication in $\mathbb F_{256}/\mathbb F_4$ based on Chudnovsky$^2$ multiplication algorithm applied on the Fermat curve $x^3+y^3=1$ defined over $\mathbb F_4$. Our method of construction leads to a multiplication algorithm in $\mathbb F_{256}/\mathbb F_4$ having a lower scalar complexity with an optimal bilinear complexity.

\subsubsection{Experiment of Baum-Shokrollahi}\label{BSExperiment}

The article \cite{bash} presents Chudnovsky$^2$ multiplication in $\mathbb F_{4^4}$, for the case $q=4$ and $n=4$.
The elements of $\mathbb F_4$ are denoted by $0,1,\omega$ and $\omega^2$.  The algorithm construction requires the use of an elliptic curve over $\mathbb F_4$ with at least 9 $\mathbb F_4$-rational points (which is the maximum possible number by Hasse-Weil Bound).
Note that in this case, Conditions $1)$ and $2)$ of Theorem \ref{theoprinc} are well satisfied. It is well known that the Fermat curve $u^3+v^3=1$ satisfies this condition.
By the substitutions $x=1/(u+v)$ and $y=u/(u+v)$, we get the isomorphic curve $y^2+y=x^3+1$.
From now on, $F/\mathbb F_q$ denotes the algebraic function field associated to the elliptic curve $\mathcal C$ with plane model $y^2+y=x^3+1$, of genus one. The projective coordinates $(x:y:z)$ of $\mathbb F_4$-rational points of this elliptic curve are:
\begin{gather*}
 P_\infty =(0:1:0), P_1=(0:\omega:1), P_2=(0:\omega^2:1), P_3=(1:0:1),\\
P_4=(1:1:1),P_5=(\omega:0:1), P_6=(\omega:1:1),P_7=(\omega^2:0:1),P_8=(\omega^2:1:1).
\end{gather*}

Now, we represent $\mathbb F_{256}$ as $\F_4[x]/\mathcal Q(x)$ with primitive root $\alpha$, where $\mathcal Q(x)= x^4+x^3+\omega x^2+\omega x+\omega$.

\begin{itemize}
  \item For the place $Q$ of degree 4, the authors considered $Q=\sum_{i=1}^{4}\mathfrak{p}_i$ where $\mathfrak{p}_1$ corresponds
  to the $\mathbb F_{4^4}$-rational point with projective coordinates $(\alpha^{16}: \alpha^{174}:1)$ and $\mathfrak{p}_2, \mathfrak{p}_3, \mathfrak{p}_4$
  are its conjugates under the Frobenius map. We see that $\alpha^{16}$ is a root of the irreducible polynomial $\mathcal Q(x)= x^4+x^3+\omega x^2+\omega x+\omega$.
  Thus,  the place $Q$ is a place lying over the place $(\mathcal Q(x))$ of $\mathbb F_4(x)/\mathbb F_4$. Note also that the place
  $((\mathcal Q(x))$ of $\mathbb F_4(x)/\mathbb F_4$ is totally splitted in the algebraic function field $F/\mathbb F_4$, which means that there exist two places of
  degree $n$ in $F/\mathbb F_4$ lying over the place $(\mathcal Q(x))$ of $\mathbb F_4(x)/\mathbb F_4$, since the function field $F/\mathbb F_q$
  is an extension of degree $2$ of the rational function field $\mathbb F_4(x)/\mathbb F_q$.
  The place $Q$ is one of the two places in $F/\mathbb F_4$ lying over the place $(\mathcal Q(x))$. Notice that the second place is given by the orbit of the conjugated point $(\alpha^{16}: \alpha^{174}+1: 1)$. Therefore, we can represent $\mathbb F_{256}=\mathbb F_{4^4}=\F_4[x]/\mathcal Q(x)$ as the residue class field $F_Q$ of the place $Q$ in $F/\F_4$.

  \item For the divisor $D$, we choose the place described as $\sum_{i=1}^{4}\mathfrak {d}_i$ where $\mathfrak {d}_1$ corresponds to the $\mathbb F_{4^4}$-rational point $(\alpha^{17}: \alpha^{14}:1)$ and $\mathfrak{d}_2, \mathfrak{d}_3, \mathfrak{d}_4$ are its conjugates under the Frobenius map. By computation we see that $\alpha^{17}$ is a root of irreducible polynomial $\mathcal D(x)=x^2+x+\omega$ and $\deg D=4$ because  $\mathfrak {d}_1$, $\mathfrak{d}_2, \mathfrak{d}_3, \mathfrak{d}_4$ are all distinct. Therefore, $D$ is the only place in $F/\mathbb F_4$ lying over the place $(\mathcal D(x))$ of $\mathbb F_4(x)$ since the residue class field $F_D$ of the place $D$ is a quadratic extension of the residue class field $F_{\mathcal D}$ of the place $\mathcal D$,  which is an inert place of $\mathbb F_4(x)$ in $F/\F_4$.
\end{itemize}
The matrix $T_{2D}$ obtained in the basis of Riemann-Roch space $L(2D)$: \\$ {\mathcal B}_{2D}= \{f_1=1/f,f_2=x/f,f_3=y/f,f_4=x^2/f, f_5=1/f^2,f_6=xy/f^2,f_7=y/f^2,f_8=x/f^2\},$ with $f=x^2+x+\omega$ is the following:
$$T_{2D}=\left(
    \begin{array}{cccccccc}
      0 & 0 & 0 & 1 & 0 & 0 & 0 & 0 \\
      \omega^2 & 0 & 1 & 0 & \omega & 0 & \omega^2 & 0 \\
      \omega^2 & 0 & \omega & 0 & \omega & 0 & 1 & 0 \\
      \omega^2 & \omega^2 & 0 & \omega^2 & \omega & 0 & 0 & \omega \\
      \omega^2& \omega^2 & \omega^2 & \omega^2 & \omega & \omega & \omega & \omega \\
      \omega & \omega^2 & 0 & 1 & \omega^2 & 0 & 0 & 1 \\
      \omega & \omega^2 & \omega & 1 & \omega^2 & 1 & \omega^2 & 1 \\
      \omega & 1 & 0 & \omega^2 & \omega^2 & 0 & 0 & \omega\\
    \end{array}
  \right).$$
Then, computation gives:
$$\hspace{0.2cm} C=\left(
    \begin{array}{cccccccc}
      1 & 0 & 0 & 0 & \omega & 0 & \omega^2 & \omega \\
      0& 1 & 0 & 0 & 0 & \omega^2 & \omega & 0 \\
      0& 0 & 1& 0  & 1 & 0 & 0 & 1  \\
      0& 0 & 0 & 1 & 1 & \omega & 0  & \omega\\
    \end{array}
  \right)$$
and
$$ CT_{2D}^{-1}=
\left(\begin{array}{*{20}{c}}
1&\omega&1&\omega & 1 & 1 & \omega & 0\\
1&0&\omega^2&\omega & 1 & \omega^2 & 1 & \omega \\
1&\omega&\omega&\omega^2 & 1 & \omega^2  & \omega &  \omega \\
0&\omega &\omega^2&\omega & 1 & \omega^2 & 0 & 0 \\
\end{array}\right).$$
Consequently,  we obtain:
$$N_{z}(T_{D}) = 10, \;\; N_{z}(CT_{2D}^{-1}) =5.$$ and $$N_{1}(T_{D}) = 5, \;\; N_{1}(CT_{2D}^{-1}) =10.$$
Thus, we have the following quantities: $\mu_{s,0}(\mathcal{U}^{F,n}_{D,Q,\mathcal{P}})=71$ by Formula \eqref{ARz}, 
$\mu_{s,1}(\mathcal{U}^{F,n}_{D,Q,\mathcal{P}})=76$ by Formula \eqref{AR1} and 
finally $\mu_{s}(\mathcal{U}^{F,n}_{D,Q,\mathcal{P}})=51$ by Formula \eqref{scalar}.

\subsubsection{New designs of the Baum-Shokrollahi Construction (BSC)}

In this section, we follow the approach described previously
and we improve the Chudnovsky$^2$ multiplication algorithm in $\mathbb F_{4^4}$ constructed 
by Baum and Shokrollahi in \cite{bash}. By using the same elliptic curve and the same set $\{D,Q,\mathcal{P}\}$ 
(up to a permutation of the set ${\mathcal P}$ since it has no influence on scalar resp. bilinear complexity 
by Section \ref{strategiecomplete}), we obtain an algorithm with the same bilinear complexity and lower scalar complexity. 
The new construction of CCMA for the multiplication in $\mathbb F_{256}/\mathbb F_4$ is based 
upon complexity analysis in Section \ref{anacomp} and the strategies highlighted in Section \ref{DQfixed}.

\vspace{1em}

{\bf a) Optimization with Algorithm \ref{setup2}}

\vspace{1em}

By using  Algorithm \ref{setup1} (taking into account uniquely 
the optimization of the number of zeros) applied on the same set 
$\{F/\F_q, D,Q,{\mathcal P}\}$ used in Section \ref{BSExperiment} (up to a permutation of the set ${\mathcal P}$), we obtain the following basis
$${\mathcal B}^{opt}_{2D}={\mathcal B}_{D,max}\cup {\mathcal B}^c_D$$
of $\mathcal L(2D)$, where ${\mathcal B}_{D,max}=\{f_1,f_2,f_3,f_4\}$ and ${\mathcal B}^c_D=\{f_5,f_6,f_7,f_8\}$ with:

\begin{align*}
f_1&=(\omega x^2 + x)/(x^2 + x + \omega),\\
f_2&= (\omega^2x^2 + \omega^2 x + \omega^2)/(x^2 + x + \omega),\\
f_3&=  \omega^2y/(x^2 + x + \omega)+ (\omega^2 x + 1)/(x^2 + x + \omega),\\
f_4&=  \omega^2y/(x^2 + x + \omega) + (\omega^2 x + \omega)/(x^2 + x + \omega),\\
f_5&= (x^2 + x)y/(x^4 + x^2 + \omega^2) + (x^4 + \omega x^3 + \omega x^2 + \omega x)/(x^4 + x^2 + \omega^2),\\
f_6&=  \omega^2 xy/(x^4 + x^2 + \omega^2) + (\omega x^4 + x^2 + \omega x + 1)/(x^4 + x^2 + \omega^2),\\
f_7&=  (\omega^2 x + 1)y/(x^4 + x^2 + \omega^2) + (\omega^2 x^4 + \omega^2 x^3 + \omega x^2 + \omega)/(x^4 + x^2 +
    \omega^2),\\
f_8&=  (x^2 + \omega x + 1)y/(x^4 + x^2 + \omega^2)  + (x^4 + \omega x^3 + x^2 + \omega^2 x + \omega^2)/(x^4 +
    x^2 + \omega^2).
\end{align*}

In this basis, we obtained the matrice $T_{2D}$ of the second evaluation map $Ev_{\mathcal{P}}$, 
where $\mathcal P:=\{P_{\infty}, P_1, P_2,P_7, P_8,P_3,P_4,P_5\}$ is the ordered set of rational 
places used in CCMA:

$$T_{2D}=\left(
    \begin{array}{cccccccc}
    \omega &\omega^2 &0&   0&   1&   \omega& \omega^2 & 1\\
   0 & \omega &0  & \omega & 0&  \omega&   0&   \omega \\
 0 &  \omega &\omega & 0 & 0 &  \omega &\omega & 0 \\
1&0&0&1&1&1&\omega^2 &\omega^2\\
1&0&1&0&\omega &\omega & \omega^2 &0\\
0 &0&1&0& \omega &\omega &0&1\\
 0&0&0&1&1&\omega^2 &\omega &0\\
\omega   &\omega   &1   &\omega^2   &1&0&0&\omega^2 \\
    \end{array}
  \right)$$ and
  $$T^{-1}_{2D,4}=\left(
    \begin{array}{cccccccc}
    0 &\omega &1&0&0&1&1&\omega^2\\
  0  & 0&0&0  & 1& \omega &\omega & \omega^2\\
 \omega^2 &\omega &\omega^2 &\omega^2&   \omega& \omega & 0&0\\
1 &\omega^2   &\omega &\omega^2 & 0   &0   &1&\omega^2  \\
 \end{array}
  \right).$$

Therefore, $N_{z}(T_{D})= 16$ and $N_{z}(T^{-1}_{2D,4})=11$. 
Note that without taking into account the optimization criterium mentioned in Remark \ref{optimizationcriteriumNzthenN1}, 
we have: $N_{1}(T_{D})= 7$ and $N_{1}(T^{-1}_{2D,4})=6$. So, we obtain $\mu_{s,0}(\mathcal{U}^{F,n}_{D,Q,\mathcal{P}})=53$
(a gain of $25\%$ with respect to BSC). 
Finally, if we compute the other quantities, we obtain $\mu_{s,1}(\mathcal{U}^{F,n}_{D,Q,\mathcal{P}})=76$ 
(equality with BSC) and  $\mu_{s}(\mathcal{U}^{F,n}_{D,Q,\mathcal{P}})=33$ 
(a gain of 54,5\% with respect to BSC).

\vspace{1em}

{\bf b) Optimization with Algorithm \ref{setup3}}

\vspace{1em}

By using  Algorithm \ref{setup3} (taking into account uniquely 
the optimization of the number of zeros) applied on the same set 
$\{F/\F_q, D,Q,{\mathcal P}\}$ used in Section \ref{BSExperiment} (up to a permutation of the set ${\mathcal P}$), we obtain the following basis
$${\mathcal B}^{opt}_{2D}={\mathcal B}_{D,max}\cup {\mathcal B}^c_D$$
of $\mathcal L(2D)$, where ${\mathcal B}_{D,max}=\{f_1,f_2,f_3,f_4\}$ and ${\mathcal B}^c_D=\{f_5,f_6,f_7,f_8\}$ with:

\begin{align*}
f_1&=(y + \omega x+\omega^2)/(x^2 + x + \omega),\\
f_2&= (y + \omega^2 x + \omega)/(x^2 + x + \omega, \\
f_3&= (\omega x^2 + \omega^2 x)/(x^2 + x + \omega),\\
f_4&= (\omega y)/(x^2 + x + \omega),\\
f_5&= (\omega x^2 + \omega x)y+\omega^2 x^4 + \omega x^3 +  x^2 +  x+ \omega)/(x^4 + x^2 + \omega^2),\\
f_6&=  (\omega^2 x^2y+\omega x^4 + \omega x^3+ x^2 + \omega x )/(x^4 + x^2 + \omega^2),\\
f_7&=  (x^2+ \omega^2 x)y + \omega x^4 + \omega x^2 )/x^4 + x^2 +
    \omega^2),\\
f_8&=  (\omega x + \omega)y + \omega x^4)/(x^4 +
    x^2 + \omega^2).
\end{align*}

In this basis, we obtained the matrice $T_{2D}$ of the second evaluation map $Ev_{\mathcal{P}}$, where 
$\mathcal P:=\{P_{\infty}, P_1,P_2,P_7,P_8,P_3,$ $P_4,P_5\}$ is the ordered set of rational places used in CCMA:

$$T_{2D}=\left(
    \begin{array}{cccccccc}
0 & 0 &  \omega &  0 &\omega^2 &  \omega &  \omega &  \omega\\
\omega^2 &  0 & 0 &  \omega &\omega^2 &  0 &  0 &  1\\
0& \omega^2 & 0 &\omega^2 &\omega^2 &0 &  0 &  \omega\\
\omega^2& \omega^2 &\omega^2 &0 &  1 &  1 &  0 &\omega^2\\
0&0 &\omega^2 &  1 &  1 &  0 &\omega^2 &\omega^2\\
0 &  1 & 0 & 0 &  0 &\omega^2 &  1 &  \omega\\
 \omega &\omega^2 & 0 &\omega^2 &  1  & \omega &  0 &  1\\
\omega^2 & 0 &  \omega  & 0 & \omega &  0 &  1 &\omega^2\\
   \end{array}
\right )$$
and
  $$T^{-1}_{2D,4}=\left(
    \begin{array}{cccccccc}
     1 & 0 & 0 &\omega^2 &  \omega & 0& \omega^2 &\omega^2\\
 1 &  1 &\omega^2 &0 & 0 &\omega^2 & 0 &  1\\
0 &  1 &  \omega &  1 &\omega^2 &  \omega & 0 & 0\\
\omega^2 &  \omega &\omega^2 & 0 & \omega & 0 &\omega^2 & 0
 \end{array}
  \right).$$

Therefore, $N_{z}(T_{D})= 16$ and $N_{z}(T^{-1}_{2D,4})=12$. 
Note that without taking into account the optimization criterium mentioned in Remark \ref{optimizationcriteriumNzthenN1}, 
we have: $N_{1}(T_{D})= 2$ and $N_{1}(T^{-1}_{2D,4})=6$. 
So, we obtain $\mu_{s,0}(\mathcal{U}^{F,n}_{D,Q,\mathcal{P}})=52$ (a gain of $27\%$ over BSC). 
Note also that we improve the result obtained in \cite{baboda} ($+2\%$). 
Finally, if we compute the other quantities, we obtain $\mu_{s,1}(\mathcal{U}^{F,n}_{D,Q,\mathcal{P}})=86$ 
( a loss of 13\% with respect to BSC) and  $\mu_{s}(\mathcal{U}^{F,n}_{D,Q,\mathcal{P}})=42$ 
(a gain of 21,5\% with respect to BSC).

\begin{remark}

Regarding the total scalar complexity, we notice that a worse result is obtained using Algorithm \ref{setup3} 
than using Algorithm \ref{setup2}. However, this is not significant because we did not take into account 
the optimization criterion for the number of 1, wishing to focus on the optimization of the number of zeros. 
It is therefore likely to obtain even better constructions, by using the criteria mentioned in Remark \ref{Remark2Setup2}.

\end{remark}


\begin{thebibliography}{10}
\bibitem{atbaboro2}
Kevin Atighehchi, St\'ephane Ballet, Alexis Bonnecaze, and Robert Rolland.
\newblock {Arithmetic in Finite Fields based on Chudnovsky's multiplication
  algorithm}.
\newblock {\em Mathematics of Computation}, 86(308):2977--3000, 2017.

\bibitem{ball1}
St\'ephane Ballet.
\newblock {C}urves with {M}any {P}oints and {M}ultiplication {C}omplexity in
  {A}ny {E}xtension of $\mathbb{F}_q$.
\newblock {\em {F}inite {F}ields and {T}heir {A}pplications}, 5:364--377, 1999.

\bibitem{ball2}
St\'ephane Ballet.
\newblock {Q}uasi-optimal {A}lgorithms for {M}ultiplication in the {E}xtensions
  of $\mathbb{F}_{16}$ of degree $13$, $14$, and $15$.
\newblock {\em {J}ournal of {P}ure and {A}pplied {A}lgebra}, 171:149--164,
  2002.

\bibitem{baboda}
St\'ephane Ballet, Alexis Bonnecaze, and Thanh-Hung Dang.
\newblock On the scalar complexity of chudnovsky$^2$ multiplication algorithm
  in finite fields.
\newblock In {\em {CAI'19}}, volume 11545 of {\em Lecture Notes in Computer
  Science}, pages 64--75. Springer, 2019.

\bibitem{bachpirararo}
St\'ephane Ballet, Jean Chaumine, Julia Pieltant, Matthieu Rambaud, Hugues
  Randriambololona, and Robert Rolland.
\newblock On the tensor rank of multiplication in finite extensions of finite
  fields and related issues in algebraic geometry.
\newblock {\em {U}spekhi {M}atematicheskikh {N}auk ({R}ussian {M}athematical
  {S}urveys)}, to appear.

\bibitem{bash}
Ulrich Baum and Amin Shokrollahi.
\newblock An optimal algorithm for multiplication in
  $\mathbb{F}_{256}/\mathbb{F}_4$.
\newblock {\em {A}pplicable {A}lgebra in {E}ngineering, {C}ommunication and
  {C}omputing}, 2(1):15--20, 1991.

\bibitem{chau1}
Jean Chaumine.
\newblock On the bilinear complexity of multiplication in small finite fields.
\newblock {\em {C}omptes {R}endus de l'{A}cad\'emie des {S}ciences, S\'erie I},
  343:265--266, 2006.

\bibitem{chch}
David Chudnovsky and Gregory Chudnovsky.
\newblock Algebraic complexities and algebraic curves over finite fields.
\newblock {\em {J}ournal of {C}omplexity}, 4:285--316, 1988.

\bibitem{groo}
Hans De~Groote.
\newblock Characterization of division algebras of minimal rank and the
  structure of their algorithm varieties.
\newblock {\em {S}{I}{A}{M} {J}ournal on {C}omputing}, 12(1):101--117, 1983.

\bibitem{piel}
Julia Pieltant.
\newblock {\em {T}ours de corps de fonctions alg\'ebriques et rang de tenseur
  de la multiplication dans les corps finis}.
\newblock PhD thesis, Universit\'e d'{A}ix-{M}arseille, {I}nstitut de
  {M}ath\'ematiques de {L}uminy, 2012.

\bibitem{shok}
Amin Shokhrollahi.
\newblock Optimal algorithms for multiplication in certain finite fields using
  algebraic curves.
\newblock {\em {S}{I}{A}{M} {J}ournal on {C}omputing}, 21(6):1193--1198, 1992.

\bibitem{stic2}
Henning Stichtenoth.
\newblock {\em Algebraic Function Fields and Codes}.
\newblock Number 254 in {G}raduate {T}exts in {M}athematics. Springer-Verlag,
  second edition, 2008.

\bibitem{wino3}
Shmuel Winograd.
\newblock {On Multiplication in Algebraic Extension Fields}.
\newblock {\em {T}heoretical {C}omputer {S}cience}, 8:359--377, 1979.

\end{thebibliography}
\end{document}